\documentclass[11pt]{article}

\usepackage{amsmath,amsthm,amscd,latexsym}
\usepackage{amsfonts,pdfsync,color,graphicx,float}
\usepackage[psamsfonts]{amssymb}
\usepackage{epsfig}
\usepackage{color}
\usepackage[noadjust]{cite}
\usepackage{accents}

\usepackage{listings}

\newcommand{\R}{{\mathbb R}}

\newcommand{\rank}{{\rm rank\;}}

\usepackage[top=1.5in,bottom=1.5in,left=1.3in,right=1.0in]{geometry}

\newtheoremstyle{mystyle}               % Name
  {}                % Space above
  {}                % Space below
  {}        % Body font
  {}                % Indent amount
  {\bfseries \itshape}       % Theorem head font
  {.}      % Punctuation after theorem head
  { }      % Space after theorem head, ' ', or \newline
  {}       % Theorem head spec

\newtheorem{theorem}{Theorem}[section]

\newtheorem{lemma}[theorem]{Lemma} 
\newtheorem{corollary}[theorem]{Corollary}
\theoremstyle{definition}

\newtheorem{example}[theorem]{Example}	
\theoremstyle{mystyle}
\newtheorem{remark}[theorem]{Remark}

\title{Green's function related to a $n$ order linear differential equation coupled to arbitrary linear non local boundary conditions}
\date{}
\author{Alberto Cabada, Luc{\' i}a L\'opez-Somoza and Mouhcine Yousfi\\
	 Instituto de  Ma\-te\-m\'a\-ti\-cas, Facultade de Matem\'aticas, \\
	Universidade de Santiago de Com\-pos\-te\-la, 15782 Santiago de Compostela,\\ Galicia, Spain.\\
	alberto.cabada@usc.es; lucia.lopez.somoza@usc.es; mouhcine.yousfi@rai.usc.es}

\begin{document}
\maketitle
\begin{abstract}
	
In this paper we obtain the explicit expression of the Green's function related to a general $n$ order differential equation coupled to non-local linear boundary conditions. In such boundary conditions, a $n$ dimensional parameter dependence is also assumed. Moreover, some comparison principles are obtained.
	
	The explicit expression depends on the value of the Green's function related to the two-point homogeneous problem, that is, we are assuming that when all the  parameters involved on the boundary conditions take the value zero then the problem has a unique solution which is characterized by the corresponding Green's function $g$. The expression of the Green's function $G$ of the general problem is given as a function of $g$ and the real parameters considered at the boundary conditions.
		
	It is important  to show that, in order to ensure the uniqueness of solutions of the linear considered problem, we must assume a non resonant additional condition on the considered problem, which depends on the  non local conditions and the corresponding parameters.
	
		We point out that the assumption of the uniqueness of solutions of the two-point homogeneous problem is not a necessary condition to ensure the solution of the general case. Of course, in this situation the expression we are looking for must be obtained in a different manner. 
		
	To show the applicability of the obtained results, a particular example is given.
	\end{abstract}

\section{Introduction}

Most of the real phenomena that appear in fields as, among others, Physics, Engineering, Biology or Medicine, are modeled by Ordinary Differential Equations coupled to suitable boundary conditions located at some given set of the interval of definition. The majority of them take values at the extremes of the interval, they are known as two-point boundary value problems. There are a long tradition in studying these kind of problems and a lot of works in this direction have been developed to ensure the existence, uniqueness or multiplicity of solutions, coupled to their stability or instability (see, for instance, \cite{ARG, CH}).

To allow, on the boundary conditions, suitable dependence at some fixed points (or sets) of the interval, that are not the extreme ones, permits the study of a wider set of problems that model suitable real phenomena. Therefore, the so-called non-local conditions allow us to deal with more complicated problems that model more difficult real phenomena.  In the non resonance case, such kind of problems can be studied as an equivalent integral equation of the type

\begin{equation*}\label{pham-gijw}
	u(t)=r(t)\, B(u)+\int_a^b k(t,s)f(s,u(s))\,ds,\quad t\in [a,b],
\end{equation*}
where $r$ is a continuous function, $B:C([a,b]) \to \R$ is a continuous linear functional, $k$ is the Green's function related to the considered problem, and $f(t,x)$ is the nonlinear part of the considered equation.

This kind of equations cover different non-local situations as, for instance, to model the steady-state of a heated bar of length $b-a$ subject to a thermostat, where a controller in one end adds or removes heat accordingly to the temperature measured by a sensor at a point of the bar. This type of heat-flow problem has been studied in several works on the literature, see \cite{guimer, gijwems, gi-pp-ft, jw-narwa} and references therein.

An important part of the used methods to ensure the existence of solutions are mainly related to the theory of lower and upper solutions \cite{CH}, degree theory \cite{CIT, gi-pp-ft, gijwems} or monotone iterative techniques \cite{A}. In all of these cases it is fundamental to ensure the constant sign (in the whole square of definition or in a suitable subset) of the Green's function related to the considered problem. In many situations, this study is not trivial and requires many tedious and complicate calculations. Such difficulty increases with the non-local operators on the boundary. One of the most common non-local boundary conditions are given as integral equations (some of them in the Stieltjes sense) and has been applied to different situations as fourth order beam equations \cite{CJ}, second order  problems \cite{HY} or fractional equations \cite{CW, DNS}.

To be concise, in this paper we will consider the following  $n$-th order linear boundary value problem with parameter dependence:
\begin{equation}\label{1}
\left\{
\begin{array}{rlll}
T_n\left[M\right]u\left(t\right)&=&\sigma\left(t\right),&\quad t\in I:=\left[a,b\right],\\
B_{i}\left(u\right)&=&\delta_{i}\, C_i\left(u\right),&\quad i=1,\ldots,n,
\end{array}
\right.
\end{equation}
where
\begin{equation*}
T_n\left[M\right]u\left(t\right):=L_{n} u\left(t\right)+M\,u(t), \quad t \in I,
\end{equation*}
with 
\begin{equation*}
L_{n} u\left(t\right):=u^{\left(n\right)}\left(t\right)+a_{1}\left(t\right) u^{\left(n-1\right)}\left(t\right)+\cdots +a_{n}\left(t\right) u\left(t\right),\quad t\in I.
\end{equation*}
Here   $\sigma$ and $a_{k}$ are continuous functions for all $k=0,\ldots ,n-1$, $M\in \mathbb{R}$ and $\delta_{i}\in \mathbb{R}$ for all $i=1,\ldots,n$.

$C_i:C(I)\rightarrow \mathbb{R}$ is a linear continuous operator and $B_i$ covers the general two point linear boundary conditions, i.e.:
\begin{equation*}
B_{i}\left(u\right)=\displaystyle \sum_{j=0}^{n-1} \left(\alpha_{j}^{i} u^{\left(j\right)}\left(a\right)+\beta_{j}^{i} u^{\left(j\right)}\left(b\right) \right),\quad i=1,\ldots ,n,
\end{equation*}
being $\alpha_{j}^{i},\;\; \beta_{j}^{i}$  real constants for all  $i=1,\ldots ,n,\;\; j=0,\ldots ,n-1$.

\begin{remark}
	Examples of operator $C_i$ can be the integral operator 
	\begin{center}
	$C_i\left(u\right)=\displaystyle \int_{J} u\left(s\right) v\left(s\right)ds,\quad J\subset I \mbox{ ($J$ an interval)}\quad$ with $\quad v\in C\left(J\right)$,
	\end{center}
	or the multi-point operator
	\begin{equation*}
	C_i\left(u\right)=\displaystyle \sum_{k=1}^{r} \epsilon_{k} u\left(\nu_{k}\right),\quad \nu_{k}\in I,\quad \epsilon_{k}\in \mathbb{R},\quad k=0,\ldots,r.
	\end{equation*}
\end{remark}

We point out that Problem~\eqref{1} covers any $n$-th order differential equation and that on the choice of $\delta_i$ any of them could vanish, so it may be though as a perturbation of a two-point boundary value problem.

So, by considering the following homogeneous problem related to the general equation \eqref{1}:
\begin{equation}\label{2}
	\left\{
	\begin{aligned}
	T_n\left[M\right] u\left(t\right)&=0,\quad t\in I,\\
	B_{i}\left(u\right)&=0,\quad i=1,\ldots n,
	\end{aligned}
	\right.
	\end{equation}
 we will obtain the explicit expression of the Green's function related to the non-local problem~\eqref{1}  under the assumption that the corresponding homogeneous Problem~\eqref{2} has only the trivial solution. Moreover we will characterize the spectrum of Problem~\eqref{1} as a  function of the value of the non-local operators over functions related to the Green's function of Problem~\eqref{1}.
 
 We notice that the non-local linear operators depend only on the values of the function that we are looking for, but they could depend on any of its derivatives, and by using analogous reasoning, the result that we could obtain would be similar. 
 
 The paper is organized as follows. In next section we obtain the expression of the Green's function related to Problem~\eqref{1} and characterize its spectrum. In Section 3 we present an example where the formula is used to obtain the corresponding expression and to describe the exact set of parameters for which its Green's function has constant sign on $I \times I$.

% Remember the following result of Fredholm's Alternative that descusses the existence and uniqueness of the solution of Problem~\eqref{1}.
% \begin{theorem}
% 	The problem 
% 	\begin{equation*}
% 	T_n\left[M\right] u\left(t\right)=\sigma\left(t\right),\quad t\in I=\left[a,b\right],\quad B_{i}\left(u\right)=\delta_{i} C\left(u\right),\quad i=1,\ldots ,n,
% 	\end{equation*}
% 	has a unique solution if and only if the homogeneous problem 
% 	\begin{equation}
% 	T_n\left[M\right] u\left(t\right)=0,\quad t\in I=\left[a,b\right],\quad B_{i}\left(u\right)=0,\quad i=1,\ldots ,n,
% 	\end{equation}
% 	has only the trivial solution $u=0$.
% \end{theorem}

\section{Explicit expression of the solution of Problem~\eqref{1}}

This section is devoted to deduce the explicit expression of the solution of general Problem~\eqref{1}. To this end, we assume that the homogeneous Problem~\eqref{2} has as unique solution the trivial one. In such a case, it is very well known that Problem~\eqref{1}, with $\delta_i=0$, $i=0, \ldots,n$, has a unique solution for any $\sigma \in C(I)$ given. Moreover, such solution is given by
\begin{equation}\label{e-v}
	v\left(t\right)=\displaystyle \int_{a}^{b} g_{M}\left(t,s
	\right) \sigma\left(s\right) ds.
\end{equation}
Here $g_{M}$ denotes the Green's function related to Problem~\eqref{2}, which exists and is unique (see, for details, \cite{A, kar}). 

Now, we enunciate the following particular case of the result proved in \cite[page 35]{A}:

\begin{theorem}
	\label{t-ex-un-Green}
	The following boundary value problem
	\begin{equation}\label{e-n-or-gen}
		T_n[M]\,u(t)=\sigma(t),\;  t\in I,\quad B_i(u)=h_i,\; i=1,\ldots,n,
	\end{equation}
has a unique solution for any $\sigma \in C(I)$ and $h_i \in \R$, $i\in \{0, 1, \ldots,n\}$, if and only if 
\begin{equation}
	\label{det}
\det \begin{pmatrix}
	B_1(u_1)&\ldots&B_1(u_n)\\
	\vdots&\ddots&\vdots\\
	B_n(u_1)&\ldots&B_n(u_n)
\end{pmatrix} \neq 0,
\end{equation}
where $ \bigl(u_1, \ldots, u_n \bigr) $ is any set of
linearly independent solutions of $ T_n [M] \, u (t) = 0 $.
\end{theorem}

\begin{remark}
It is immediate to verify that the fact that the determinant in previous result is different from zero does not depend on the chosen set of linearly independent solutions.
\end{remark}

\begin{remark}
	Notice that condition \eqref{det} is independent of the non homogeneous part of Problem~\eqref{e-n-or-gen}: $\sigma$ and $h_i$, $i \in \{1, \ldots,n\}$.
\end{remark}

\begin{remark}
		One can see in \cite[page 35]{A} that the following property	
	\begin{equation}
		\label{e-rango-escalar}
		\rank \left(\begin{array}{ccc|ccc}
			\alpha_0^1&\ldots&\alpha_{n-1}^1&\;\beta_0^1&\ldots&\beta_{n-1}^1\\
			\vdots&\ddots&\vdots&\vdots&\ddots&\vdots\\
			\alpha_0^n&\ldots&\alpha_{n-1}^n&\;\beta_0^n&\ldots&\beta_{n-1}^n
		\end{array}\right)=n,
	\end{equation}
	is a necessary (but not sufficient) condition  to ensure the uniqueness of solution of Problem~\eqref{e-n-or-gen}.
\end{remark}

As a direct consequence of Theorem \ref{t-ex-un-Green} we deduce the following result

\begin{lemma}\label{l-ex-green}
There exists the unique Green's function related to Problem~\eqref{2}, $g_M$, if and only if for any $i\in \{1, \cdots,n\}$, the following problem 
\begin{equation}\label{e-wj}
	\left\{
	\begin{aligned}
		T_n\left[M\right] u\left(t\right)&=0,\quad t\in I,\\
		B_{j}\left(u\right)&=0,\quad j\neq i,\\
		B_{i}\left(u\right)&=1,
	\end{aligned}
	\right.
\end{equation}
has a unique solution, that we denote as $\omega_{i}\left(t\right)$, $t\in I$.
\end{lemma}

In the following result, under suitable assumptions concerning the spectrum of the considered problem, we prove the existence and uniqueness of the solution of Problem~\eqref{1}. Moreover, the expression of its related Green's function is obtained. 

\begin{theorem}\label{th_Ci}
	Assume that Problem~\eqref{2} has $u=0$ as its unique solution and let $g_M$ be its related Green's function. 
	Let $\sigma \in C\left(I\right)$, and $\delta_{i},$  $i=1,\dots, n$, be such that
	\begin{equation}
		\label{e-espectro}
		 \det(I_n-A)\neq 0,
	\end{equation}
with $I_n$ the identity matrix of order $n$ and $A=(a_{ij})_{n\times n}\in \mathcal{M}_{n\times n}$ given by \[a_{ij}=\delta_{j}\,C_i(\omega_{j}), \quad i,\; j \in \{1, \ldots,n\}.\]
Then  Problem~\eqref{1} has a unique solution $u\in C^n\left(I\right)$, given by the expression
	\begin{equation}
		\label{e-u}
	u\left(t\right)=\displaystyle \int_{a}^{b} G\left(t,s,\delta_{1},\ldots,\delta_{n},M\right) \sigma\left(s\right) ds,
	\end{equation}
	where 
	\begin{equation}\label{e-G-delta}
	G\left(t,s,\delta_{1},\ldots,\delta_{n},M\right):=g_{M}\left(t,s\right)+ \sum_{i=1}^{n} \sum_{j=1}^{n} \delta_{i} \, b_{ij} \, \omega_{i}(t)  \, C_j\left(g_{M}(\cdot,s)\right),\quad t, \; s \in I,
	\end{equation}
with $\omega_j$ defined on Lemma \ref{l-ex-green} and  $B=\left(b_{ij}\right)_{n\times n}=(I_n-A)^{-1}$.
\end{theorem}
\begin{proof}
Since Problem~\eqref{1} has a unique solution when  $\delta_{i}=0$ for all $i=1,\ldots,n$, from Lemma~\ref{l-ex-green} we know that any solution of \eqref{1} satisfies the following expression
	\begin{equation}\label{6}
		u\left(t\right)=v\left(t\right)+\displaystyle \sum_{i=1}^{n} \omega_{i}\left(t\right) \delta_{i} \, C_i\left(u\right),\quad t\in I,
	\end{equation}
with $v$ given by \eqref{e-v}.

Applying linear continuous operators $C_j$ on both sides of \eqref{6} we infer that
\[C_j\left(u\right)=C_j\left(v\right)+C_j\left(\displaystyle \sum_{i=1}^{n} \delta_{i} \,\omega_{i}\,C\left(u\right)\right) =C_j\left(v\right)+\displaystyle \sum_{i=1}^{n} \delta_{i} \, C_j\left(\omega_{i}\right) C_i\left(u\right), \quad j=1,\dots,n, \]
from which we deduce that
\begin{equation*}
\begin{split}
C_j(u) -\sum_{i=1}^{n} \delta_{i}\, C_j(\omega_{i})\, C_i(u) =C_j(v), \quad j=1,\dots,n. 
\end{split}
\end{equation*}
Therefore, we arrive at the following systems of equations
\begin{equation}\label{eq:syst_I_A}
(I_n-A) \, \left(\begin{array}{c}
C_1(u) \\ C_2(u) \\ \vdots \\ C_n(u)
\end{array}\right) = \left(\begin{array}{c}
C_1(v) \\ C_2(v) \\ \vdots \\ C_n(v)
\end{array}\right).
\end{equation}

From previous equality we deduce that
\[C_i(u)=\sum_{j=1}^{n} b_{ij}\, C_j(v), \quad i=1,\dots,n, \]
and substituting this expression in \eqref{6} we obtain that
\[u(t)= v(t)+ \sum_{i=1}^{n} \delta_i\, \omega_{i}(t) \left(\sum_{j=1}^{n} b_{ij}\, C_j(v) \right), \quad t \in I. \]
To calculate $C_j\left(v\right)$ we use the fact that $C_j$ is linear and continuous, so we get that
	\begin{equation*}
	C_j\left(v\right)=C_j\left(\displaystyle \int_{a}^{b} g_M\left(\cdot,s\right) \sigma(s)\, ds\right)=\displaystyle \int_{a}^{b} C_j\left(g_{M}\left(\cdot,s\right) \right)  \sigma(s) \, ds.
	\end{equation*}
Using the previous equality, we have that
		\begin{equation*}
		\begin{aligned}
		u\left(t\right)&=\displaystyle \int_{a}^{b} g_{M}\left(t,s\right) \sigma\left(s\right) ds+ \displaystyle \sum_{i=1}^{n} \delta_{i} \, \omega_{i}(t) \left(\sum_{j=1}^{n} b_{ij} \int_{a}^{b} C_j\left(g_{M}\left(\cdot,s\right) \right)  \sigma(s) \, ds  \right) \\
		&=\displaystyle \int_{a}^{b} \left(g_{M}\left(t,s\right)+ \sum_{i=1}^{n} \delta_{i} \, \omega_{i}(t) \left(\sum_{j=1}^{n} b_{ij} \, C_j\left(g_{M}(\cdot,s)\right) \right) \right) \sigma(s) \, ds \\
		&=\displaystyle \int_{a}^{b} \left(g_{M}\left(t,s\right)+ \sum_{i=1}^{n} \sum_{j=1}^{n} \delta_{i} \, b_{ij} \, \omega_{i}(t)  \, C_j\left(g_{M}(\cdot,s)\right) \right) \sigma(s) \, ds \\
		&=\displaystyle \int_{a}^{b} G\left(t,s,\delta_{1},\ldots,\delta_{n},M\right) \sigma\left(s\right) \, ds.
		\end{aligned}
		\end{equation*}
	
	So, we have proved that under assumption \eqref{e-espectro} coupled to the uniqueness of solution of Problem~\eqref{2}, Problem~\eqref{1} has at least one solution given by expression \eqref{e-u}.

To conclude the proof, we must show the uniqueness of the solution. To this end, suppose that $u$ and $v$ are two different solutions of Problem~\eqref{1}. Then,
\begin{equation}\label{sour}
\left\{
\begin{array}{rlll}
T_n\left[M\right]\left(u-v\right)\left(t\right)&=&0,& t\in I,\\
B_{i}\left(u-v\right)&=& \delta_{i} \, C_i\left(u-v\right),& i=1,\ldots,n.
\end{array}
\right.
\end{equation}

As a consequence, we have that  
\begin{equation*}
\left(u-v\right)\left(t\right)=\displaystyle \sum_{i=1}^{n} \delta_{i} \, C_i\left(u-v\right) \omega_{i}\left(t\right), \quad t \in I.
\end{equation*}

Applying operator $C_j$ in both sides again, we have that 
\begin{equation*}
C_j\left(u-v\right)=\displaystyle \sum_{i=1}^{n} \delta_{i} \, C_i\left(u-v\right) C_j\left(\omega_{i}\right), \quad j=1,\dots,n,
\end{equation*}
or, which is the same,
\begin{equation*}
	(I_n-A) \, \left(\begin{array}{c}
		C_1(u-v) \\ C_2(u-v) \\ \vdots \\ C_n(u-v)
	\end{array}\right) = \left(\begin{array}{c}
		0 \\ 0 \\ \vdots \\ 0
	\end{array}\right).
\end{equation*}

Condition \eqref{e-espectro} implies that $C_i\left(u-v\right)=0$ for $i=1,\dots,n$. 
Hence, form \eqref{sour} we deduce that $u-v$ is a solution of the homogeneous problem 
\begin{equation*}
\left\{
\begin{aligned}
T_n\left[M\right]\left(u-v\right)\left(t\right)& =0,\;\; t\in I,\\
B_{i}\left(u-v\right)& = 0,\;\; i=1,\ldots,n.
\end{aligned}
\right.
\end{equation*} 
Since this problem has only the trivial solution, we deduce that $u=v$ on $I$, and the proof is concluded.
\end{proof}

\begin{remark}
	We notice that in previous result, we assume that there is a unique Green's function related to Problem~\eqref{2}. Such condition does not depend on $\delta_i$ or operators $C_i$, $i=1, \ldots,n$. It is obvious that this condition is fundamental to construct function $G$ on \eqref{e-G-delta}. However, such condition is not necessary in order to deduce the existence and uniqueness of solution of Problem~\eqref{1}. In practical situation our hypotheses ensure the existence of a unique solution of Problem~\eqref{1} provided for any parameters $(M, \delta_1, \ldots,\delta_n)$, such that $M$ is not an eigenvalue of Problem~\eqref{2}. But, as we will see in next section, this condition is not necessary and Problem~\eqref{2} could have a unique solution for some choice of $(M, \delta_1, \ldots,\delta_n)$, with $M$ an eigenvalue of Problem~\eqref{2}.
	
	Moreover, we assume the non spectral condition \eqref{e-espectro}, which is equivalent to assume that $1$ is not an eigenvalue of matrix $A$. When such condition fails we have that Problem~\eqref{1} has not a unique solution. So, this non spectral condition characterizes the uniqueness of solution of Problem \eqref{1} provided the existence of $g_M$ is assumed. In case of $M$ being an eigenvalue of Problem~\eqref{2}, condition \eqref{e-espectro} has no sense because $\omega_i$ do not exist.	
\end{remark}

We will show now the particular case of considering that all the functionals at boundary conditions are the same (that is, there is some linear continuous operator $C$ such that $C_i=C$ for $i=1,\dots,n$). In this case, since we have only $C(u)$ as a unique unknown variable, system \eqref{eq:syst_I_A} reduces to the one dimensional equation
\begin{equation}
	\label{e-C}
	\left(1-\sum_{i=1}^{n}\delta_{i}\, C(\omega_{i}) \right) C(u)=C(v), 
\end{equation}
%so \[I_n-A=\left(1-\displaystyle\sum_{i=1}^{n}\delta_{i}\, C(\omega_{i})\right) \in \mathcal{M}_{1\times 1}\]
and condition \eqref{e-espectro} reduces to
\begin{equation}
	\label{e-espectro2}
	\sum_{i=1}^{n} \delta_{i} \, C(\omega_{i})\neq 1.
\end{equation}

In which case, it is obvious that 

\begin{equation}
	\label{e-C(u)}
	C(u)= \frac{C(v)}{1-\displaystyle\sum_{i=1}^{n}\delta_{i}\, C(\omega_{i})}.
\end{equation}

As a direct consequence, we obtain the following result for this particular case.
\begin{corollary}\label{31}
	Assume that Problem~\eqref{2} has $u=0$ as its unique solution and let $g_M$ be its unique Green's function. 
	Let $\sigma \in C\left(I\right)$, and $\delta_{i},$  $i=1,\dots, n$, be such that \eqref{e-espectro2} holds. 	Then problem
	\begin{equation}\label{1-C}
		\left\{
		\begin{array}{rlll}
			T_n\left[M\right]u\left(t\right)&=&\sigma\left(t\right),&\quad t\in I,\\
			B_{i}\left(u\right)&=&\delta_{i}\, C\left(u\right),&\quad i=1,\ldots,n,
		\end{array}
		\right.
	\end{equation}
has a unique solution $u\in C^n\left(I\right)$, given by the expression
	\begin{equation*}
		u\left(t\right)=\displaystyle \int_{a}^{b} G\left(t,s,\delta_{1},\ldots,\delta_{n},M\right) \sigma\left(s\right) ds,
	\end{equation*}
	where 
	\begin{equation}\label{e-G-delta2}
		G\left(t,s,\delta_{1},\ldots,\delta_{n},M\right):=g_{M}\left(t,s\right)+\dfrac{\displaystyle \sum_{i=1}^{n} \delta_{i} \omega_{i}\left(t\right)}{1-\displaystyle \sum_{j=1}^{n} \delta_{j} C\left(\omega_{j}\right)} C\left(g_{M}\left(\cdot,s\right)\right).
	\end{equation}
\end{corollary}
\begin{proof}
It is enough to show that in this case expression \eqref{e-G-delta} can be rewritten as \eqref{e-G-delta2}.

Indeed, since we have a unique functional $C$ (and so, the sum in $j$ reduces to a unique term), it is clear that we can argue as in the proof of Theorem \ref{th_Ci}, by denoting  
$$I_n-A\equiv (a_{11})=\left(1-\displaystyle\sum_{i=1}^{n}\delta_{i}\, C(\omega_{i})\right)$$
and
\[B=(I_n-A)^{-1}\equiv (b_{11}) =\left(\frac{1}{1-\displaystyle\sum_{i=1}^{n}\delta_{i}\, C(\omega_{i})}\right).\]

As a consequence, we deduce that expression \eqref{e-G-delta} is rewritten in this case as
\begin{equation*}\begin{split}
G\left(t,s,\delta_{1},\ldots,\delta_{n},M\right)&=g_{M}\left(t,s\right)+ \sum_{i=1}^{n} \delta_{i} \, b_{11} \, \omega_{i}(t)  \, C\left(g_{M}(\cdot,s)\right)\\
&=g_{M}\left(t,s\right)+\dfrac{\displaystyle \sum_{i=1}^{n} \delta_{i} \omega_{i}\left(t\right)}{1-\displaystyle \sum_{j=1}^{n} \delta_{j} C\left(\omega_{j}\right)} C\left(g_{M}\left(\cdot,s\right)\right).
\end{split}\end{equation*}

\end{proof}

	\begin{example}
	If operator $C$  is given by $C\left(u\right)=\displaystyle \int_{a}^{b} u\left(s\right) ds$, then
	\begin{equation*}
		\begin{aligned}
			C\left(\displaystyle \int_{a}^{b} g_{M}\left(t,s\right) \sigma\left(s\right) ds \right)&=\displaystyle \int_{a}^{b} \left(\displaystyle \int_{a}^{b} g_{M}\left(t,s\right) \sigma\left(s\right) ds\right) dt \\
			&=\displaystyle \int_{a}^{b} \left(\displaystyle \int_{a}^{b} g_{M}\left(t,s\right) dt\right) \sigma\left(s\right) ds=\displaystyle \int_{a}^{b} C\left(g_{M}\left(\cdot,s\right)\right) \sigma\left(s\right) ds.
		\end{aligned}
	\end{equation*}
\end{example}

	\begin{example}
	If $C$ is defined as $C\left(u\right)=u(c)$, $c \in (a,b)$, we have that
	\begin{equation*}
		\begin{aligned}
			C\left(\displaystyle \int_{a}^{b} g_{M}\left(t,s\right) \sigma\left(s\right) ds \right)&=\displaystyle \int_{a}^{b} g_{M}\left(c,s\right) \sigma\left(s\right) ds =\displaystyle \int_{a}^{b} C\left(g_{M}\left(\cdot,s\right)\right) \sigma\left(s\right) ds.
		\end{aligned}
	\end{equation*}
\end{example}

As a direct consequence of expression \eqref{e-G-delta2}, we deduce the following comparison result:
\begin{corollary}\label{c-comparison}
	Assume that Problem~\eqref{2} has $u=0$ as its unique solution and let $g_M$ be its unique Green's function. 
Assume that condition \eqref{e-espectro2} holds and let $G$ be the Green's function related to Problem \eqref{1-C}. 
	Suppose that the following hypotheses are satisfied:
	\begin{itemize}
		\item[(a)] $\displaystyle \sum_{j=1}^{n} \delta_{j} C\left(\omega_{j}\right) <1$.
		\item[(b)] $\delta_{i} \omega_{i}\left(t\right)\geq 0$, $\forall t\in I,\quad i=1,\dots,n$.
		\item[(c)] If $u\geq 0$ on $I$, then $C\left(u\right)\geq 0$.
	\end{itemize}
Then, the following assertions are fulfilled:
	\begin{itemize}
		\item[(i)] If $g_{M}\left(t,s\right)\geq 0$, for all $\left(t,s\right)$\ on $I\times I$ then  $G\left(t,s,\delta_{1},\ldots,\delta_{n},M\right)\geq g_{M}\left(t,s\right)\geq 0$ for all $\left(t,s\right)$ on $I\times I$.
				\item[(ii)] If $g_{M}\left(t,s\right)\leq 0$, for all $\left(t,s\right)$\ on $I\times I$ then  $G\left(t,s,\delta_{1},\ldots,\delta_{n},M\right)\leq g_{M}\left(t,s\right)\leq 0$ for all $\left(t,s\right)$ on $I\times I$.
\end{itemize}
\end{corollary}
\begin{remark}
	As we will see in next section, the conditions of  previous corollary are sufficient but not necessary to ensure the positiveness of the related Green's function.
\end{remark}

Now, given $M\in \mathbb{R}$ and $\delta_{j}, j\neq k$ be fixed, by differentiating equality \eqref{e-G-delta2} with respect to $\delta_{k}$ we deduce that
\begin{equation}\label{33}
	\dfrac{\partial G}{\partial \delta_{k}}\left(t,s,\delta_{1},\delta_{2},\cdots,\delta_{n},M\right)=\dfrac{\omega_{k}\left(t\right) \left(1-\displaystyle \sum_{j=1}^{n} \delta_{j} C\left(\omega_{j}\right)\right)+C\left(\omega_{k}\right) \displaystyle \sum_{j=1}^{n} \delta_{j} \omega_{j}\left(t\right)} {\left(1-\displaystyle \sum_{j=1}^{n} \delta_{j} C\left(\omega_{j}\right) \right)^{2}} C\left(g_{M}\left(\cdot,s\right) \right).
\end{equation}
Thus, we can study  the monotony of the Green's function related to Problem \eqref{1-C}  with respect to any parameter $\delta_{k}$.

\section{First order periodic problem}
This section is devoted to show the applicability of the expression \eqref{e-G-delta2} obtained in previous section. Moreover, we show the validity of the assumptions of Theorem~\ref{th_Ci} and Corollary~\ref{c-comparison}.

To be concise, we study the sign of the Green's function related to the following perturbed first order periodic problem.
\begin{equation}\label{e-per}
\left\{
\begin{array}{rlll}
u'\left(t\right)+M u\left(t\right)&=&\sigma\left(t\right),& t\in \left[0,1\right],\\
u\left(0\right)-u\left(1\right)&=&\delta \displaystyle \int_{0}^{1} u\left(s\right)ds,
\end{array}
\right.
\end{equation}
with $M, \delta \in \mathbb{R}$.

It is immediate to verify that  the spectrum of Problem~\eqref{e-per} is given by 
\begin{equation*}
	\Lambda_{M\delta}=\{\left(M,M\right),\;\; M\in \mathbb{R}\}.  
\end{equation*}

In particular, when we consider the homogeneous periodic problem ($\delta=0$):
\begin{equation}\label{16}
	\left\{
	\begin{aligned}
		u'\left(t\right)+M u\left(t\right)&=\sigma\left(t\right),\;\; t\in \left[0,1\right],\\
		u\left(0\right)-u\left(1\right)&=0,
	\end{aligned}
	\right.
\end{equation}
 we have that $M=0$ is the unique eigenvalue of the considered problem. That is, there is a unique $g_M$ if and only if $M \neq 0$. Moreover, see \cite{A}, it is immediate to verify that the expression of the  Green's function of Problem~\eqref{16} is given by the following expression 
  \begin{equation}\label{e-gM}
 	g_{M}\left(t,s\right)=\frac{1}{1-e^{-M}}\left\{
 	\begin{array}{lll}
 		e^{-M\left(t-s\right)}, &\;\;0\leq s\leq t\leq 1,\\[2pt]
 		 e^{-M\left(t-s+1\right)},&\;\; 0<t<s\leq 1.
 	\end{array}
 	\right.
 \end{equation}

Using the notations on Lemma \ref{l-ex-green}, it is not difficult to verify that, see \cite{A}, that 
\[\omega_{1}\left(t\right)=g_{M}\left(t,0\right)=\frac{e^{-M\,t}}{1-e^{-M}}, \quad t \in I.
\]

As a consequence, in this case, condition \eqref{e-espectro2} is written as $\delta\neq M$, $ M\not = 0$. Thus, formula \eqref{e-G-delta2} can be applied to this set of $(M,\delta)$. We point out that, in this case, it is valid for all the values $\left(M,\delta\right)$ that are not on the spectrum of \eqref{e-per} except the ones given by $(0,\delta)$, with $\delta \not = 0$. The expression for this last situation must be studied separately.

First, we deduce the following symmetric property of the Green's function related to Problem~\eqref{e-per}.

\begin{lemma}
	\label{G-symmetric}
	Assume that Problem~\eqref{e-per} has a unique solution and let $	G\left(t,s,\delta,M\right)$ be its related Green's function. Then the following symmetric property holds:	
	\begin{equation}\label{dadah}
		G\left(t,s,\delta,M\right)=-G\left(1-t,1-s,-\delta,-M\right).
	\end{equation} 
\end{lemma}

\begin{proof}
	Let 
	\[u(t)=\int_0^1{G\left(t,s,\delta,M\right)\, \sigma(s)\, ds}\]
	be the unique solution of Problem~\eqref{e-per}.
	
	It is immediate to verify that $v(t):=u(1-t)$ is the unique solution of the following problem:
	
	\begin{equation*}
		\left\{
		\begin{array}{rlll}
			v'\left(t\right)-M v\left(t\right)&=&-\sigma\left(1-t\right),& t\in \left[0,1\right],\\
			v\left(0\right)-v\left(1\right)&=&-\delta \displaystyle \int_{0}^{1} v\left(s\right)ds.
		\end{array}
		\right.
	\end{equation*}

As a direct consequence, we deduce that
	\[v(t)=-\int_0^1{G\left(t,s,-\delta,-M\right)\, \sigma(1-s)\, ds}.\]

On the other hand, we have
\begin{equation*}
	v(t)=u(1-t)=\int_0^1{G\left(1-t,s,\delta,M\right)\, \sigma(s)\, ds}= \int_0^1{G\left(1-t,1-s,\delta,M\right)\, \sigma(1-s)\, ds}.
\end{equation*}
Therefore, the equality \eqref{dadah} is fulfilled directly by identifying the two previous equalities.
\end{proof}

Therefore it is enough to study the sign of the Green's function $G\left(t,s,\delta,M\right)$ for $M>0$ and $\delta\not = M$ (the case $M=0$ and $\delta\neq0$ will be considered further).

In our case, the expression \eqref{e-G-delta2} is given by the following formula:

\begin{equation*}\label{34}
	G\left(t,s,\delta,M\right)=g_{M}\left(t,s\right)+ \dfrac{\delta g_{M}\left(t,0\right)}{1-\delta \displaystyle \int_{0}^{1} g_{M}\left(t,0\right) dt} \displaystyle \int_{0}^{1} g_{M}\left(t,s\right) dt. 
\end{equation*}

One can verify that
\[C(g_M(\cdot,s)):=\int_{0}^{1} g_{M}\left(t,s\right) dt=\frac{1}{M} \quad \mbox{ for all $s\in\left[0,1\right]$ and $M \neq 0$}.\]

Thus, we arrive at the following explicit expression of the Green's function $G$:

\begin{equation}\label{e-G-per}
	G\left(t,s,\delta,M\right)=\frac{\delta}{M-\delta} \dfrac{e^{-M t}}{1-e^{-M}}+\frac{1}{1-e^{-M}} \left\{
	\begin{array}{ll}
		\displaystyle e^{-M\left(t-s\right)} ,&\;\; 0\leq s\leq t \leq 1,\\[2pt]
		\displaystyle e^{-M\left(t-s+1\right)},&\;\; 0<t<s\leq 1.
	\end{array}
	\right.
\end{equation}

\begin{remark}
We point out that, since for all $s \in (0,1)$ it is verified that
\[
\lim_{t \to s^+}g_M(t,s)\equiv g_M(s^+,s)= 1 + \lim_{t \to s^-}g_M(t,s)\equiv 1+g_M(s^-,s),
\]
we can define $g_M(s,s)$ as $g_M(s^+,s)$ or $g_M(s^-,s)$ at our convenience. 

This is valid too for function $G(t,s,\delta,M)$ and thus  equality \eqref{dadah} in Lemma \ref{G-symmetric} must be interpreted in this sense.

As we will see in the sequel, this fact has no influence on the sign of the Green's function.

\end{remark}

It is immediate to check that $g_{M}\left(t,s\right)>0$ for all $t,s\in I$ and $M>0$. So, using  Corollary~\ref{c-comparison} we have that $G\left(t,s,\delta,M\right)>0$ for all $\left(t,s\right)\in I\times I$ and $0\le\delta<M$.

Now let's see the range of $\delta<0 $  and $M>0$ for which  function $G\left(t,s,\delta,M\right)$ is positive on $I \times I$.

Since $	G\left(t,s,0,M\right)=g_{M}\left(t,s\right)>0$ for all $t, \, s \in I$ and $M>0$, we know, from \eqref{e-G-per}, that  the Green's function $G\left(t,s,\delta,M\right)$ will be positive for some values of $\delta<0$.

Moreover, 
\begin{equation}\label{derivada-G-delta}
\frac{\partial G}{\partial \delta}	\left(t,s,\delta,M\right)=\frac{M}{(M-\delta)^2} \dfrac{e^{-M t}}{1-e^{-M}}>0, \quad \mbox{for all $M>0$, $\delta \neq M$ and $t, s \in I$}.
\end{equation}

 As a consequence, for any $M>0$ fixed, the Green's function $G$ is strictly increasing with respect to $\delta$ and so we have that  the optimal value, $\delta(M)$, will be either $-\infty$ or the biggest negative real value for which $G\left(t,s,\delta,M\right)$ attains the value zero at some point  $(t_0,s_0) \in I \times I$.
 
To obtain this optimal value, we must take into account that, by equation  \eqref{e-G-per}, we have that for any $s \in (0,1)$, the Green's function $G \in C^1([0,s) \cup (s,1])$ and there are two real constants, $k_1(s,\delta)$ and $k_2(s,\delta)$, such that
\begin{equation}\label{e-G-k1}
G\left(t,s,\delta,M\right)=k_1(s,\delta) e^{M \, t}, \quad \mbox{for all $t \in [0,s)$}
	\end{equation}
	and
	\begin{equation}\label{e-G-k2}
		G\left(t,s,\delta,M\right)=k_2(s,\delta) e^{M \, t}, \quad \mbox{for all $t \in [s,1]$}.
	\end{equation}
	
	So, we deduce that if $G\left(t_0,s_0,\delta,M\right) =0$ for some $t_0 \in [0,s_0)$ ( $t_0 \in [s_0,1]$) then it is fulfilled that $G\left(t,s_0,\delta,M\right) =0$ for all $t \in [0,s_0)$ ($t \in [s_0,1]$).

	Moreover, since for any $s \in (0,1)$, the Green's function satisfies the boundary condition:
	
	\begin{equation}\label{G0-G1}
	G\left(0,s,\delta,M\right)-G\left(1,s,\delta,M\right)=\delta \int_0^1{G\left(t,s,\delta,M\right)\,dt},
	\end{equation}
		we deduce that, whenever $G>0$ on $I \times I$ and $\delta<0$, it holds that $G\left(0,s,\delta,M\right)<G\left(1,s,\delta,M\right)$.
		
		Thus, we must look for the biggest value of $\delta <0$ for which 
		
	\[
	h(s) :=G\left(0,s,\delta,M\right)=\frac{\delta  e^M+(M-\delta ) e^{M s}}{\left(e^M-1\right) (M-\delta )}=0\quad \mbox{ at some $s \in (0,1)$.} 
	\]
		
		Since 
		\[
		h'(s)= \frac{M e^{M s}}{e^M-1}>0\quad \mbox{ for all  $s \in I$ and $M>0$,} 
		\]
		we conclude that the optimal value of $\delta$ comes from the first root of the equation
	\[
h(0) =\frac{\delta  e^M+(M-\delta )}{\left(e^M-1\right) (M-\delta )}=0,
\]
which, trivially, is given by
\begin{equation*}
\delta=\dfrac{M}{1-e^{M}}.
\end{equation*}

So, we have obtained the following result.

\begin{lemma}
	\label{l-G-per>0} Let $M>0$, then the Green's function related to Problem~\eqref{e-per} is strictly positive on $I \times I$ if and only if 
	\[
	\delta \in \left(\dfrac{M}{1-e^{M}}, M\right).
	\]
	Moreover, if 	$\delta = \dfrac{M}{1-e^{M}}$ then $	G\left(t,s,\delta,M\right)=0$ for all $t \in [0,s)$ and $	G\left(t,s,\delta,M\right)>0$ for all $t \in [s,1]$.
\end{lemma}

To study the values for which $G\left(t,s,\delta,M\right)<0$ on $I\times I$, we can make an analogous argument. In this case, we know that if the set is not empty, necessarily  $\delta>M>0$.

Now, using equation \eqref{G0-G1}, we have that, if $G<0$ on $I \times I$ and $\delta>0$ then $G\left(0,s,\delta,M\right)<G\left(1,s,\delta,M\right)$. So we must look for the first zero of

\[
k(s):= G\left(1,s,\delta,M\right)=\frac{\delta +(M-\delta ) e^{M s}}{\left(e^M-1\right) (M-\delta )}.
\]

		Since 
\[
k'(s)= \frac{M e^{M s}}{e^M-1}>0\quad \mbox{ for all  $s \in I$ and $M>0$,} 
\]
we conclude that the optimal value of $\delta$ comes from the first root of the equation
\[
k(1) =\frac{\delta +(M-\delta ) e^{M}}{\left(e^M-1\right) (M-\delta )}=0,
\]
which, trivially, is given by
\begin{equation*}
	\delta=\dfrac{M\, e^M}{e^{M}-1}.
\end{equation*}

This way, we have obtained the following result.

\begin{lemma}
	\label{l-G-per<0}
	Assume that $M>0$, then the Green's function related to Problem~\eqref{e-per} is strictly negative on $I \times I$ if and only if 
	\[
	\delta \in \left(M, \dfrac{M\, e^M}{1-e^{M}}\right).
	\]
	Moreover, if 	$\delta = \dfrac{M\, e^M}{1-e^{M}}$ then $	G\left(t,s,\delta,M\right)=0$ for all $t \in [s,1]$ and $	G\left(t,s,\delta,M\right)<0$ for all $t \in [0,s)$.
\end{lemma}

For $M<0$, using the property of symmetry \eqref{dadah} it follows that:
\begin{lemma}
	\label{l-G-per-M<0} If $M<0$ the following properties are fulfilled:
	\begin{enumerate}
		\item $G\left(t,s,\delta,M\right)<0$ for all $\left(t,s\right)\in I \times I$, if and only if $\delta \in \left(M,\dfrac{M e^{M}}{e^{M}-1}\right)$.
		\item If 	$\delta = \dfrac{M\,e^{M}}{1-e^{M}}$ then $	G\left(t,s,\delta,M\right)=0$ for all $t \in  [s,1]$ and $	G\left(t,s,\delta,M\right)<0$ for all $t \in [0,s)$.
		\item  $G\left(t,s,\delta,M\right)>0$ for all $\left(t,s\right)\in I \times I$, if and only if $\delta \in \left(\dfrac{M}{1-e^{M}},M\right)$.  
		\item If 	$\delta = \dfrac{M}{1-e^{M}}$ then $	G\left(t,s,\delta,M\right)=0$ for all $t \in [0,s)$ and $	G\left(t,s,\delta,M\right)>0$ for all $t \in  [s,1] $.
		
	\end{enumerate}

\end{lemma}

The case with $M=0$ is not included in formula \eqref{e-G-per} because $M=0$ is an eigenvalue of Problem~\eqref{16}. It is not difficult to verify that the solution of Problem~\eqref{e-per} for $M=0$ is given by
\begin{equation*}
u\left(t\right)=\displaystyle \int_{0}^{1} G\left(t,s,\delta\right) \sigma(s)\,ds,
\end{equation*}
where 
\begin{equation*}
G\left(t,s,\delta\right)=\left\{
\begin{array}{ll}
s-\frac{1}{\delta} ,&\;\; 0\leq s\leq t \leq 1,\\\\
s-\frac{1}{\delta}-1,&\;\; 0<t<s\leq 1.
\end{array}
\right.
\end{equation*}

As a consequence, we have that $G\left(t,s,\delta\right)<0$ for all $t,s\in I$ if and only if $0<\delta<1$, and that $G\left(t,s,\delta\right)>0$ for all $t,s\in \left[0,1\right]$ if and only if $-1<\delta<0$.

Moreover 
\begin{center}
$G(1,1,1)=0$ and $	G(t,s,1)<0$ for all $(t,s) \in (I\times I)\backslash\{(1,1)\}$ 
\end{center}
and
\begin{center}
$G(0,0,-1)=0$ and $	G(t,s,-1)>0$ for all $(t,s) \in (I\times I)\backslash\{(0,0)\}$.
\end{center}

Figure~\ref{Fig:reg_signo_cte} shows the regions where the function $G$ maintains a constant sign.
\bigskip
\begin{figure}[H]
\begin{center}
	\includegraphics[width=6cm]{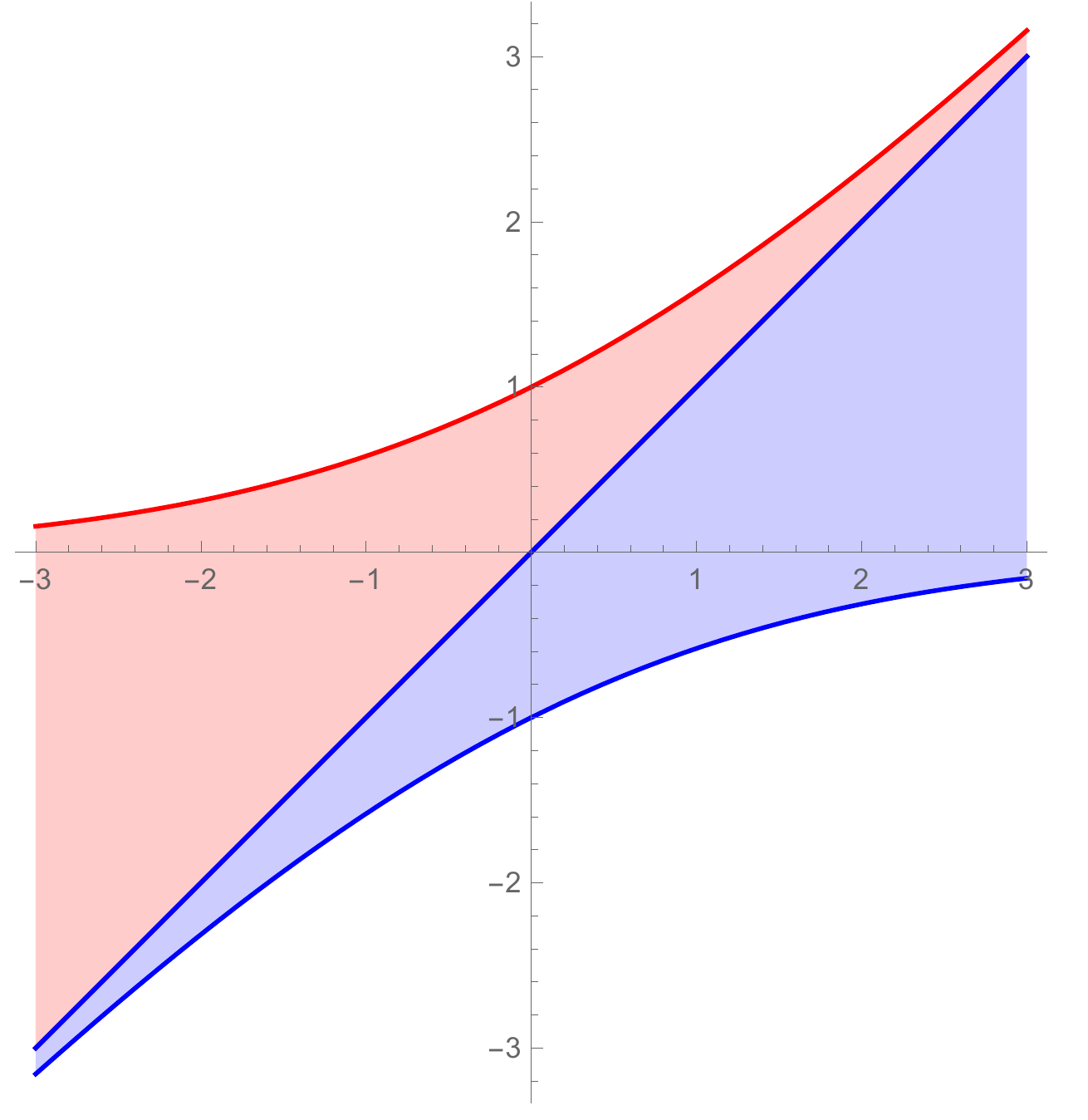}
\caption{Regions of positive and negative sign of the Green's function in the plane $M-\delta$. The blue region represents the positive sign of the Green function $G$ while the red region corresponds to the negative sign of Green's function.} \label{Fig:reg_signo_cte}
\end{center}
\end{figure}

\end{document}